\RequirePackage{fix-cm}
\documentclass[smallextended]{svjour3}       % onecolumn (second format)
\smartqed  % flush right qed marks, e.g. at end of proof
\usepackage{graphicx}
\usepackage{amssymb}
\usepackage{amsmath}

\usepackage{dsfont}
\usepackage{enumerate}
\journalname{}
\begin{document}

\title{Explicit expressions for higher order convolutions of Cauchy numbers}

%\titlerunning{Monotone and fast computation of Euler's constant}        % if too long for running head

\author{Jos\'{e} A. Adell         \and
        Alberto Lekuona
}


\institute{Jos\'{e} A. Adell \at
              Departamento de M\'{e}todos Estad\'{\i}sticos, Facultad de
Ciencias, Universidad de Zaragoza, 50009 Zaragoza (Spain)\\
              \email{adell@unizar.es}           %  \\
           \and
           Alberto Lekuona \at
             Departamento de M\'{e}todos Estad\'{\i}sticos, Facultad de
Ciencias, Universidad de Zaragoza, 50009 Zaragoza (Spain)\\
              \email{lekuona@unizar.es}
}

\date{Received: date / Accepted: date}

\maketitle

\begin{abstract}
We give explicit expressions for higher order convolutions of Cauchy numbers, either as one single integral or in terms of the Stirling numbers of the first and second kinds.

% \PACS{}
\end{abstract}

\section{Introduction and main result}

The Cauchy numbers $\boldsymbol{c}=(c_n)_{n\geq 0}$ are defined (see, for instance, Comtet \cite[Ch. VII]{Com1974} or Merlini \textit{et al.} \cite{MerSprVer2006}) via their generating function as
\begin{equation}\label{eq1.1}
    G(\boldsymbol{c},z)=\dfrac{z}{\log (1+z)}=\sum_{n=0}^\infty c_n \dfrac{z^n}{n!},\quad z\in \mathds{C},\quad |z|<1,
\end{equation}
or in integral form as
\begin{equation}\label{eq1.2}
    c_n=\int_0^1 (\theta)_n\, d\theta,\quad n=0,1,\ldots ,
\end{equation}
where $(\theta)_n$ is the descending factorial, i.e., $(\theta)_n=\theta (\theta-1)\cdots (\theta -n+1)$, $n=1,2,\ldots$, $(\theta)_0=1$. Different generalizations of these numbers can be found in Komatsu and Yuan \cite{KomYua2017}, Pyo \textit{et al.} \cite{PyoKimRim2018}, and the references therein.

The starting point of this note is a recent paper by Komatsu and Simsek \cite{KomSim2016}, in which these authors pose the problem of finding rational numbers $a_0,\ldots ,a_{m-1}$, such that
\begin{equation}\label{eq1.3}
    \sum_{\substack{l_1+\cdots +l_m=\mu\\l_1,\ldots ,l_m\geq 0}} \dfrac{\mu!}{l_1!\cdots l_m!} \sum_{\substack{k_1+\cdots +k_m=n\\k_1,\ldots ,k_m\geq 0}} \dfrac{n!}{k_1!\cdots k_m!}c_{k_1+l_1}\cdots c_{k_m+l_m}=\sum_{j=0}^{m-1}a_j c_{n+\mu -j}.
\end{equation}
Actually, Komatsu and Simsek \cite{KomSim2016} find explicit formulae for $m=3$ and $m=4$ using umbral calculus. Observe that the right-hand side in \eqref{eq1.3} depends on the Cauchy numbers themselves.

The aim of this note is to provide explicit expressions for the left-hand side in \eqref{eq1.3} only depending upon the classical Stirling numbers of the first and second kinds. In other words, to provide explicit formulae which are easy to compute. Our methodology, which makes use of probabilistic representations in terms of sums of independent identically distributed random variables having the uniform distribution on $[0,1]$, also allows us to write the left-hand side in \eqref{eq1.3} as one single integral. This could be useful for theoretical purposes.

Throughout this note, we will use the following notations. Let $\mathds{N}$ be the set of positive integers and $\mathds{N}_0=\mathds{N}\cup \{0\}$. Unless otherwise specified, we assume from now on that $n,\mu\in \mathds{N}_0$, $m\in \mathds{N}$, and $z\in \mathds{C}$ with $|z|<r$, where $r>0$ may change from line to line. Recall that the Stirling numbers of the first and second kind, respectively denoted by $s(n,k)$ and $S(n,k)$, $k=0,1\ldots ,n$, are defined (see, for instance, Abramowitz and Stegun \cite[p. 824]{AbrSte1964}) by
\begin{equation}\label{eq1.4}
    (x)_n=\sum_{k=0}^n s(n,k)x^k,\qquad x^n=\sum_{k=0}^n S(n,k) (x)_k,\quad x\in \mathds{R}.
\end{equation}
On the other hand, we define the spline function
\begin{equation}\label{eq1.5}
    \rho_m(\theta)=\dfrac{1}{(m-1)!} \sum_{k=0}^{m-1}\binom{m}{k}(-1)^k (\theta-k)_+^{m-1},\quad \theta \in [0,m],
\end{equation}
where $x_+=\max (x,0)$. Finally, we denote by
\begin{equation*}
    \binom{n}{k_1,\ldots ,k_m}=\dfrac{n!}{k_1!\cdots k_m!},\quad k_1,\ldots ,k_m\in \mathds{N}_0,\quad k_1+\cdots +k_m=n,
\end{equation*}
the multinomial coefficient. With these notations, our main result is the following.

\begin{theorem}\label{th1}
We have
\begin{equation}\label{eq1.6}
    \begin{split}
       & \sum_{l_1+\cdots +l_m=\mu} \binom{\mu}{l_1,\cdots ,l_m} \sum_{k_1+\cdots +k_m=n} \binom{n}{k_1,\cdots ,k_m}c_{k_1+l_1}\cdots c_{k_m+l_m}\\
        & =\int_0^m (\theta)_{\mu+n}\rho_m(\theta)\, d\theta= \sum_{k=0}^{\mu+n}\dfrac{s(\mu+n,k)S(m+k,m)}{\binom{m+k}{m}}.
    \end{split}
\end{equation}
\end{theorem}

We mention that Zhao \cite[Corollary~3.1]{Zha2009} already obtained the second equality in \eqref{eq1.6} for $\mu=0$ using the coefficients method. The proof of Theorem~\ref{th1} is given in Section~\ref{s3}. Such a proof is based on two main ingredients, namely, the notion of binomial convolution of sequences introduced in \cite{AdeLek2017} and the probabilistic representation of $S(n,k)$ in terms of moments of appropriate random variables, shown by Sun \cite{Sun2005} (see also \cite{AdeLek2017}). These tools, together with two technical lemmas concerning general sequences of numbers, which are of interest by themselves, are presented in the following section.

\section{Technical lemmas}\label{s2}

Let $\mathcal{G}$ be the set of real sequences $\boldsymbol{u}=(u_n)_{n\geq 0}$ such that $u_0\neq 0$ and
\begin{equation*}
    \sum_{n=0}^\infty |u_n|\dfrac{r^n}{n!}<\infty,
\end{equation*}
for some radius $r>0$. If $\boldsymbol{u}\in \mathcal{G}$, we denote its generating function by
\begin{equation*}
    G(\boldsymbol{u},z)=\sum_{n=0}^\infty u_n \dfrac{z^n}{n!}.
\end{equation*}
Observe that $\boldsymbol{u}$ and $G(\boldsymbol{u},z)$ determine one each other. If $\boldsymbol{u}$ and $\boldsymbol{v}$ are in $\mathcal{G}$, the binomial convolution of $\boldsymbol{u}$ and $\boldsymbol{v}$, denoted by $\boldsymbol{u}\times \boldsymbol{v}=((u\times v)_n)_{n\geq 0}$, is defined as
\begin{equation*}
    (u\times v)_n=\sum_{k=0}^n \binom{n}{k}u_kv_{n-k}.
\end{equation*}
It turns out (see \cite[Corollary~2.2]{AdeLek2017} that $(\mathcal{G},\times)$ is an abelian group with identity element $\boldsymbol{e}=(e_n)_{n\geq 0}$ given by $e_0=1$, $e_n=0$, $n\in \mathds{N}$. On the other hand (cf. \cite[Proposition~2.1]{AdeLek2017}), if $\boldsymbol{u}^{(k)}=(u_n^{(k)})_{n\geq 0}\in \mathcal{G}$, $k=1,\ldots ,m$, then $\boldsymbol{u}^{(1)}\times \cdots \times \boldsymbol{u}^{(m)}\in \mathcal{G}$ and
\begin{equation}\label{eq2.7}
    (u^{(1)}\times \cdots \times u^{(m)})_n=\sum_{j_1+\cdots +j_m=n} \binom{n}{j_1,\ldots ,j_m} u_{j_1}^{(1)}\cdots u_{j_m}^{(n)}.
\end{equation}
In addition, $\boldsymbol{u}^{(1)}\times \cdots \times \boldsymbol{u}^{(m)}$ is characterized by its generating function
\begin{equation}\label{eq2.8}
    G(\boldsymbol{u}^{(1)}\times \cdots \times \boldsymbol{u}^{(m)},z)=G(\boldsymbol{u}^{(1)},z)\cdots G(\boldsymbol{u}^{(m)},z).
\end{equation}

Let $\mathcal{G}^\star$ be the subset of $\mathcal{G}$ consisting of those $\boldsymbol{u}=(u_n)_{n\geq 0}$ such that $u_n \neq 0$, $n\in \mathds{N}_0$. If $\boldsymbol{u}\in \mathcal{G}^\star$ and $l\in \mathds{N}_0$, we denote by
\begin{equation}\label{eq2.9}
    \boldsymbol{u}(l)=(u_{l+n})_{n\geq 0}.
\end{equation}

\begin{lemma}\label{l1}
If $\boldsymbol{u}\in \mathcal{G}^\star$ and $l\in \mathds{N}_0$, then $\boldsymbol{u}(l)\in \mathcal{G}^\star$ and
\begin{equation*}
    G(\boldsymbol{u}(l),z)=G^{(l)}(\boldsymbol{u},z).
\end{equation*}
\end{lemma}

\begin{proof}
Suppose that $G(\boldsymbol{u},z)$ is defined for $|z|<r$, for some $r>0$. Differentiation term by term gives us
\begin{equation}\label{eq2.10}
    G^{(l)}(\boldsymbol{u},z)=\sum_{n=l}^\infty u_n\, \dfrac{z^{n-l}}{(n-l)!}=G(\boldsymbol{u}(l),z),\quad |z|<r,
\end{equation}
as follows from \eqref{eq2.9}. The proof is complete.\qed
\end{proof}

\begin{lemma}\label{l2}
Let $\boldsymbol{u}^{(k)}\in \mathcal{G}^\star$, $k=1,\ldots ,m$, and $\mu \in \mathds{N}_0$. Then, $(\boldsymbol{u}^{(1)}\times \cdots \times \boldsymbol{u}^{(m)})(\mu)\in \mathcal{G}^\star$ and
\begin{equation}\label{eq2.11}
\begin{split}
    &G((\boldsymbol{u}^{(1)}\times \cdots \times \boldsymbol{u}^{(m)})(\mu),z)=G^{(\mu)}(\boldsymbol{u}^{(1)}\times \cdots \times \boldsymbol{u}^{(m)},z)\\
    &=\sum_{l_1+\cdots +l_m=\mu}\binom{\mu}{l_1,\ldots ,l_m}G(\boldsymbol{u}^{(1)}(l_1)\times \cdots \times \boldsymbol{u}^{(m)}(l_m),z).
\end{split}
\end{equation}

As a consequence,
\begin{equation}\label{eq2.12}
\begin{split}
    &(\boldsymbol{u}^{(1)}\times \cdots \times \boldsymbol{u}^{(m)})_{\mu+n}=\sum_{j_1+\cdots +j_m=\mu+n}\binom{\mu+n}{j_1,\ldots ,j_m}u_{j_1}^{(1)}\cdots u_{j_m}^{(m)}\\
    &=\sum_{l_1+\cdots +l_m=\mu}\binom{\mu}{l_1,\ldots ,l_m}\sum_{k_1+\cdots +k_m=n}\binom{n}{k_1,\ldots ,k_m} u_{k_1+l_1}^{(1)}\cdots u_{k_m+l_m}^{(m)}.
\end{split}
\end{equation}
\end{lemma}

\begin{proof}
By \eqref{eq2.7}, $\boldsymbol{u}^{(1)}\times \cdots \times \boldsymbol{u}^{(m)}\in \mathcal{G}^\star$, which implies, by virtue of Lemma~\ref{l1}, that $(\boldsymbol{u}^{(1)}\times \cdots \times \boldsymbol{u}^{(m)})(\mu)\in \mathcal{G}^\star$. Suppose that $G(\boldsymbol{u}^{(k)},z)$ is defined for $|z|<r_k$, for some $r_k>0$, $k=1,\ldots ,m$. Denote by $r=\min (r_1,\ldots ,r_m)>0$. By \eqref{eq2.8}, $G(\boldsymbol{u}^{(1)}\times \cdots \times \boldsymbol{u}^{m},z)$ is defined for $|z|<r$ and, a fortiori, so is $G((\boldsymbol{u}^{(1)}\times \cdots \times \boldsymbol{u}^{(m)})(\mu),z)$, as follows from \eqref{eq2.10}.

Hence, applying Lemma~\ref{l1} and using Leibniz's rule for differentiation in \eqref{eq2.8}, we get for $|z|<r$
\begin{equation*}
    \begin{split}
        & G((\boldsymbol{u}^{(1)}\times \cdots \times \boldsymbol{u}^{(m)})(\mu),z)= G^{(\mu)}(\boldsymbol{u}^{(1)}\times \cdots \times \boldsymbol{u}^{(m)},z) \\
         & =\sum_{l_1+\cdots +l_m=\mu} \binom{\mu}{l_1,\ldots ,l_m}G^{(l_1)}(\boldsymbol{u}^{(1)},z)\cdots G^{(l_m)}(\boldsymbol{u}^{(m)},z)\\
         &= \sum_{l_1+\cdots +l_m=\mu} \binom{\mu}{l_1,\ldots ,l_m} G(\boldsymbol{u}^{(1)}(l_1)\times \cdots \times \boldsymbol{u}^{(m)}(l_m),z),
     \end{split}
\end{equation*}
thus showing \eqref{eq2.11}. Finally, \eqref{eq2.12} is an immediate consequence of \eqref{eq2.7}, \eqref{eq2.9}, and \eqref{eq2.11}.\qed
\end{proof}

Lemma~\ref{l2} tells us that in order to compute the right-hand side in \eqref{eq2.12}, we only need to look at the $n$th coefficient in the expansion of $G^{(\mu)}(\boldsymbol{u}^{(1)}\times \cdots \times \boldsymbol{u}^{(m)},z)$. In the case at hand, that is, when $\boldsymbol{u}^{(k)}=\boldsymbol{c}$, $k=1,\ldots ,m$, such a coefficient can be described in probabilistic terms.

To this end, let $(U_j)_{j\geq 1}$ be a sequence of independent identically distributed random variables having the uniform distribution on $[0,1]$ and denote by
\begin{equation}\label{eq2.13}
    S_m=U_1+\cdots +U_m \qquad (S_0=0).
\end{equation}
We will need the following two facts. In first place, the probability density of $S_m$ is $\rho_m(\theta)$, as defined in \eqref{eq1.5} (see, for instance, Feller \cite[p. 27]{Fel1971} or Adell and Sang\"{u}esa \cite[Proposition~2.1]{AdeSan2005}). This means that, for any bounded measurable function $f:[0,m]\to \mathds{C}$, we have
\begin{equation}\label{eq2.14}
    \mathds{E}f(S_m)=\int_0^m f(\theta)\rho_m(\theta)\, d\theta,
\end{equation}
where $\mathds{E}$ stands for mathematical expectation. In second place, Sun \cite{Sun2005} (see also \cite{AdeLek2017}) showed the following probabilistic representation for the Stirling numbers of the second kind
\begin{equation}\label{eq2.15}
    S(n,m)=\binom{n}{m}\mathds{E}S_m^{n-m},\quad m=0,1\ldots ,n.
\end{equation}

\section{Proof of Theorem~\ref{th1}}\label{s3}

We will apply Lemma~\ref{l2} with $\boldsymbol{u}^{(k)}=\boldsymbol{c}$, $k=1,\ldots ,m$. In this respect, note that $\boldsymbol{c}\in \mathcal{G}^\star$, as follows from \eqref{eq1.2}. Following Pyo \textit{et al.} \cite{PyoKimRim2018}, we have from \eqref{eq1.1} and \eqref{eq2.14}
\begin{equation*}
    G(\boldsymbol{c},z)=\int_0^1 (1+z)^\theta\, d\theta=\mathds{E}(1+z)^{U_1},
\end{equation*}
since $\rho_1(\theta)=1$, $\theta \in [0,1]$, as seen from \eqref{eq1.5}. By \eqref{eq2.8}, \eqref{eq2.13}, and \eqref{eq2.14}, this implies that
\begin{equation}\label{eq3.16}
   G (\stackrel{\stackrel{m}{\smile}}{\boldsymbol{c}\times \cdots \times \boldsymbol{c}},z)= \mathds{E}(1+z)^{U_1}\cdots \mathds{E}(1+z)^{U_m}= \mathds{E}(1+z)^{S_m}=\sum_{n=0}^\infty \mathds{E}(S_m)_n \dfrac{z^n}{n!},
\end{equation}
thanks to the independence and identical distribution of the random variables involved. In turn, \eqref{eq3.16} entails that
\begin{equation*}
    G^{(\mu)} (\stackrel{\stackrel{m}{\smile}}{\boldsymbol{c}\times \cdots \times \boldsymbol{c}},z)=\sum_{n=0}^\infty \mathds{E}(S_m)_{\mu+n}\dfrac{z^n}{n!}.
\end{equation*}
We therefore conclude from Lemma~\ref{l2} and \eqref{eq2.14} that the left-hand side in \eqref{eq1.6} equals to
\begin{equation*}
    \mathds{E}(S_m)_{\mu+n}=\int_0^m (\theta)_{\mu+n}\rho_m(\theta)\, d\theta.
\end{equation*}

Finally, we get from \eqref{eq1.4} and \eqref{eq2.15}
\begin{equation*}
    \mathds{E}(S_m)_{\mu+n}=\sum_{k=0}^{\mu+n}s(\mu+n,k)\mathds{E}S_m^k=\sum_{k=0}^{\mu+n}\dfrac{s(\mu+n,k)S(m+k,m)}{\binom{m+k}{m}},
\end{equation*}
thus completing the proof of Theorem~\ref{th1}.\qed

\begin{acknowledgements}
The authors are partially supported by Research Projects DGA (E-64), MTM2015-67006-P, and by FEDER funds.
\end{acknowledgements}

\bibliographystyle{spmpsci}      % mathematics and physical sciences

 \bibliography{mybibfileAMH2018}

\begin{thebibliography}{10}
\providecommand{\url}[1]{{#1}}
\providecommand{\urlprefix}{URL }
\expandafter\ifx\csname urlstyle\endcsname\relax
  \providecommand{\doi}[1]{DOI~\discretionary{}{}{}#1}\else
  \providecommand{\doi}{DOI~\discretionary{}{}{}\begingroup
  \urlstyle{rm}\Url}\fi

\bibitem{AbrSte1964}
Abramowitz, M., Stegun, I.A.: Handbook of mathematical functions with formulas,
  graphs, and mathematical tables, \emph{National Bureau of Standards Applied
  Mathematics Series}, vol.~55.
\newblock For sale by the Superintendent of Documents, U.S. Government Printing
  Office, Washington, D.C. (1964)

\bibitem{AdeLek2017}
Adell, J.A., Lekuona, A.: Binomial convolution and transformations of {A}ppell
  polynomials.
\newblock J. Math. Anal. Appl. \textbf{456}(1), 16--33 (2017).
\newblock \urlprefix\url{https://doi.org/10.1016/j.jmaa.2017.06.077}

\bibitem{AdeSan2005}
Adell, J.A., Sang\"uesa, C.: Approximation by {$B$}-spline convolution
  operators. {A} probabilistic approach.
\newblock J. Comput. Appl. Math. \textbf{174}(1), 79--99 (2005).
\newblock \urlprefix\url{https://doi.org/10.1016/j.cam.2004.04.001}

\bibitem{Com1974}
Comtet, L.: Advanced combinatorics, enlarged edn.
\newblock D. Reidel Publishing Co., Dordrecht (1974).
\newblock The art of finite and infinite expansions

\bibitem{Fel1971}
Feller, W.: An introduction to probability theory and its applications. {V}ol.
  {II}.
\newblock Second edition. John Wiley \& Sons, Inc., New York-London-Sydney
  (1971)

\bibitem{KomSim2016}
Komatsu, T., Simsek, Y.: Third and higher order convolution identities for
  {C}auchy numbers.
\newblock Filomat \textbf{30}(4), 1053--1060 (2016).
\newblock \urlprefix\url{https://doi.org/10.2298/FIL1604053K}

\bibitem{KomYua2017}
Komatsu, T., Yuan, P.: Hypergeometric {C}auchy numbers and polynomials.
\newblock Acta Math. Hungar. \textbf{153}(2), 382--400 (2017).
\newblock \urlprefix\url{https://doi.org/10.1007/s10474-017-0744-0}

\bibitem{MerSprVer2006}
Merlini, D., Sprugnoli, R., Verri, M.C.: The {C}auchy numbers.
\newblock Discrete Math. \textbf{306}(16), 1906--1920 (2006).
\newblock \urlprefix\url{https://doi.org/10.1016/j.disc.2006.03.065}

\bibitem{PyoKimRim2018}
Pyo, S.S., Kim, T., Rim, S.H.: Degenerate {C}auchy numbers of the third kind.
\newblock J. Inequal. Appl. p. 2018:32 (2018).
\newblock \urlprefix\url{https://doi.org/10.1186/s13660-018-1626-x}

\bibitem{Sun2005}
Sun, P.: Product of uniform distribution and {S}tirling numbers of the first
  kind.
\newblock Acta Math. Sin. (Engl. Ser.) \textbf{21}(6), 1435--1442 (2005).
\newblock \urlprefix\url{https://doi.org/10.1007/s10114-005-0631-4}

\bibitem{Zha2009}
Zhao, F.Z.: Sums of products of {C}auchy numbers.
\newblock Discrete Math. \textbf{309}(12), 3830--3842 (2009).
\newblock \urlprefix\url{https://doi.org/10.1016/j.disc.2008.10.013}

\end{thebibliography}

\end{document}